\documentclass[a4paper,10pt]{article}

\usepackage{amssymb}
\usepackage{graphicx}
\usepackage{bbm}
\usepackage{mathrsfs}
\usepackage{amsmath}
\usepackage{amsthm}
\usepackage[center]{caption}
\usepackage{enumerate}
\usepackage{verbatim}
\usepackage{authblk}
\usepackage[hmargin=3cm, vmargin=3.5cm]{geometry}
\usepackage{stmaryrd}

\theoremstyle{plain}
\newtheorem{thm}{Theorem}[]

\newtheorem{lem}[thm]{Lemma}
\newtheorem{prop}[thm]{Proposition}

\theoremstyle{definition}

\renewcommand{\P}{\mathbb{P}}

\newcommand{\E}{\mathbb{E}}

\newcommand{\R}{\mathbb{R}}
\newcommand{\Z}{\mathbb{Z}}
\newcommand{\N}{\mathbb{N}}
\newcommand{\F}{\mathcal{F}}
\newcommand{\G}{\mathcal{G}}

\newcommand{\ind}{\mathbbm{1}}

\newcommand{\eps}{\varepsilon}
\newcommand{\bp}{\begin{proof}}
\newcommand{\ep}{\end{proof}}
\DeclareMathOperator{\cov}{cov}
\DeclareMathOperator{\cosech}{cosech}
\newcommand{\tcov}{T_{\cov}}
\def\bal#1\eal{\begin{align*}#1\end{align*}}

\title{Cover time for branching random walks on regular trees}
\author{Matthew I.~Roberts\thanks{University of Bath, Department of Mathematical Sciences, Bath BA2 7AY, UK. \texttt{mattiroberts@gmail.com}}}

\begin{document}
\maketitle

\begin{abstract}
Let $T$ be the regular tree in which every vertex has exactly $d\ge 3$ neighbours. Run a branching random walk on $T$, in which at each time step every particle gives birth to a random number of children with mean $d$ and finite variance, and each of these children moves independently to a uniformly chosen neighbour of its parent. We show that, starting with one particle at some vertex $0$ and conditionally on survival of the process, the time it takes for every vertex within distance $r$ of $0$ to be hit by a particle of the branching random walk is almost surely $r + \frac{2}{\log(3/2)}\log\log r + o(\log\log r)$.
\end{abstract}

\section{Introduction and main result}\label{introsec}

Consider a {\it branching random walk} (BRW) on a graph $G$, beginning with one particle at some vertex, where each particle branches into a random number of offspring (independently and according to some fixed distribution), each of which jumps to a uniformly chosen neighbour. The behaviour of BRW when $G=\Z$ is a well studied subject starting with Hammersley \cite{hammersley:postulates_subadditive}, Kingman \cite{kingman:first_birth_problem} and several papers by Biggins; see for example \cite{biggins:chernoff_in_brw, biggins:asymptotic_shape_brw, biggins:first_last_birth}. We also highlight an early paper of Bramson \cite{bramson:minimal_displacement_BRW}, which contrasts with more recent results of Aidekon \cite{aidekon:convergence_law_min_brw} and Bramson, Ding and Zeitouni \cite{bramson_ding_zeitouni:convergence_nonlattice}.

In this article we consider instead the case when the underlying graph $G$ is the regular tree in which every vertex has exactly $d\ge 3$ neighbours (of course, $G=\Z$ can be viewed as the case $d=2$). We suppose that the expected number of offspring of each particle in the branching random walk is also $d$; this is \emph{critical} in the geometric sense that the expected number of offspring moving to each neighbouring site has mean $1$. We start with one particle at the root (an arbitrary vertex) of the tree, and ask for the cover time of a ball of radius $r$. That is, how long does it take before every site within distance $r$ of the root has been visited by a particle of the BRW?

To state our result precisely, let $T$ be the infinite $d$-ary tree in which every vertex has $d$ neighbours, and fix a vertex which we label $0$ and refer to as the root or origin. Suppose that $\mu$ is a probability measure on $\Z_+$ such that $\sum_{j\ge 0}j\mu(j)=d$ and $\sum_{j\ge 0}j^2\mu(j)<\infty$. Consider a branching random walk on $T$, starting with one particle at the root, in which at every time step:
\begin{enumerate}[(a)]
\item each particle at any site $x\in T$ dies and gives birth to a random number of children independently and with distribution $\mu$;
\item each of these offspring independently jumps to a neighbour of $x$, uniformly at random.
\end{enumerate}
For each vertex $x\in T$, let $H(x)$ be the first time at which there is a particle at $x$. For $r\ge 0$, let $B(r) = \{x\in T : d(0,x)\le r\}$ and $\partial B(r) = \{x\in T : d(0,x)=\lfloor r\rfloor\}$. We are interested in the cover time of $B(r)$, defined to be
\[\tcov(r) = \max_{x\in B(r)} H(x),\]
when $r$ is large. Of course, if $\mu(0)>0$, there is a positive probability that the process will die out in finite time; however, since $\mu$ has mean larger than one and finite variance, there is strictly positive probability that the process does not die out in finite time \cite{kesten_stigum:lim_multi-dim_gw}. In this case we say that the process \emph{survives}.

\begin{thm}\label{mainthm}
For any $d\ge 3$, given that the process survives,
\[\lim_{r\to\infty} \frac{\tcov(r)-r}{\log\log r} = \frac{2}{\log(3/2)} \hspace{4mm} \hbox{almost surely.}\]
\end{thm}

This result is initially surprising for two reasons. The first is that the cover time is so close (within a constant times $\log\log r$) to its trivial lower bound of $r$. However, upon reading Bramson's article \cite{bramson:minimal_displacement_BRW}, one can see the reason for the $\log\log r$ term, and might guess convergence of the quantity in our theorem to $2/\log 2$. Indeed, in the case $d=2$, this is the correct answer. The appearance of $2/\log(3/2)$ instead comes from the fact that there are exponentially many vertices in $\partial B(r)$, some of which are hit unusually late; although it is interesting then that the answer does not depend on the value of $d\ge 3$. We give a short heuristic in Section \ref{bgsec}.

The article is set out as follows. In Section \ref{bgsec} we give some background to Theorem \ref{mainthm} as well as a heuristic walkthrough of the proof; we also state some related open problems. In Section \ref{freezing_sec} we introduce a variant of our BRW in which we freeze particles that do not move in a particular direction. We then prove the lower bound for Theorem \ref{mainthm} in Section \ref{lbsec}, and the upper bound in Section \ref{ubsec}.

\section{Background, heuristic and open questions}\label{bgsec}

We write $f(n)\asymp g(n)$ to mean that there exist constants $0<c\le C<\infty$ such that $c\le f(n)/g(n) \le C$ for all large $n$. We write $\P_{n,x}$ for the probability measure under which we start with $n$ particles at the vertex $x$. More generally, for a collection $\Gamma = (x_1,\ldots,x_n)$ of vertices, we write $\P_\Gamma$ for the probability measure under which we start with a particle at each of the vertices $x_1,\ldots,x_n$. For example $\P_{(x,x,x)} = \P_{3,x}$, and $\P = \P_{(0)} = \P_{1,0}$.

\subsection{Background}

The problem of how fast a branching random walk spreads first appeared in the mid-1970s, with papers by Hammersley \cite{hammersley:postulates_subadditive}, Kingman \cite{kingman:first_birth_problem} and Biggins \cite{biggins:chernoff_in_brw, biggins:asymptotic_shape_brw, biggins:first_last_birth} giving---amongst other results---the first-order behaviour of the particle at maximal (or minimal) distance from the origin after $n$ steps. In 1978, Bramson \cite{bramson:minimal_displacement_BRW} described a branching random walk on $\Z_+$, beginning with one particle at $0$, in which at each time step each particle branched into an average of $m>1$ new particles, each of which stayed at its previous location with probability $1/m$ and moved one step to the right with probability $1-1/m$. Letting $M_n$ be the position of the minimal particle after $n$ steps, he showed that
\begin{equation}\label{bramresult}
M_n - \Big\lceil \frac{\log\log n - \log (V+o(1))}{\log 2}\Big\rceil \to 0
\end{equation}
almost surely, where $V$ is some non-trivial random variable. One of the purposes of looking at this model was that it showed significantly different behaviour from Bramson's concurrent work on branching Brownian motion (BBM) \cite{bramson:maximal_displacement_BBM}, demonstrating that giving a result on the detailed behaviour of $M_n$ was much harder in general for BRW than BBM. In fact, results in the spirit of \cite{bramson:maximal_displacement_BBM} were not given for BRW in $\R$ until relatively recently, by Aidekon \cite{aidekon:convergence_law_min_brw} and then Bramson, Ding and Zeitouni \cite{bramson_ding_zeitouni:convergence_nonlattice}. Bramson's result \eqref{bramresult} is very closely related to the cover time problem in the case $d=2$, and we will use elements of Bramson's proof in this article.

For branching random walks on other graphs, particularly trees, much of the existing literature is concerned with recurrence and transience and related questions: see for example \cite{gantert_muller:critical_bmc_transient, liggett:BRW_contact_trees, madras_schinazi:BRW_trees, pemantle_stacey:BRW_contact_trees}. The ``maximal particle'' question mentioned above for BRW on $\R$ has no direct analogue on trees. One could ask for the maximal distance from the origin over all particles after $n$ steps; it is easy to see that with our choice of parameters ($d$-ary tree and offspring distribution mean $d$), conditionally on survival, this is $n-O(1)$ almost surely. Studying the cover time, or equivalently the largest ball that has been covered in $n$ steps, is an equally natural alternative, and with our choice of parameters it is a much more delicate and interesting question.

For our proof we will need the following well-known result on critical Galton-Watson processes, which is originally due to Kolmogorov \cite{kolmogorov:solution_biological_problem} under a third moment assumption. See for example \cite[Theorem 12.7]{lyons_peres:probability_on_trees} for a modern proof.

\begin{lem}[Kolmogorov]\label{kollem}
Suppose that $(Z_n, n\ge 0)$ is a Galton-Watson process started from $Z_0=1$ satisfying $\E[Z_1]=1$ and $\sigma^2 := \E[Z_1^2]-1 <\infty$. Then
\[n\P(Z_n>0) \to 2/\sigma^2.\]
\end{lem}

Key to our argument will be a result on the total progeny of a Galton-Watson process up to generation $n$ due to Pakes \cite{pakes:total_progeny_GW}.

\begin{lem}[Pakes]\label{pakeslem}
Suppose that $(Z_n, n\ge 0)$ is a Galton-Watson process started from $Z_0=1$ satisfying $\E[Z_1]=1$ and $\sigma^2 := \E[Z_1^2]-1 <\infty$. Let $S_n = \sum_{i=0}^n Z_i$. Then for any $\gamma\in(0,\infty)$,
\[\P(S_n\ge \gamma n^2 \,|\, Z_n>0) \to 1-F(\gamma)\]
where $F$ satisfies
\[\int_0^\infty e^{-\theta v} dF(v) = \sqrt{2\sigma^2\theta} \cosech\big(\sqrt{2\sigma^2\theta}\big) \,\,\,\, \text{for all } \theta\in[0,\infty).\]
\end{lem}

We will not need the precise form of $F$, only that $F(\gamma)$ is strictly smaller than $1$ for each finite $\gamma$. As a result, combining Lemmas \ref{kollem} and \ref{pakeslem} (using the same notation), and noting that 
\[\P(S_n\ge \gamma n^2) \ge \P(\{S_n\ge \gamma n^2\}\cap\{Z_n>0\}) = \P(S_n\ge \gamma n^2 \,|\, Z_n>0)\P(Z_n>0),\]
we obtain that for each $\gamma\in[0,\infty)$ there exists a constant $q(\gamma)>0$ depending only on $\gamma$ and $\sigma^2$ such that
\begin{equation}\label{totalprogeq}
\liminf_{n\to\infty} n\P(S_n\ge \gamma n^2) \ge q(\gamma).
\end{equation}
We will also use the following simple and well-known Chernoff bound. Suppose that $X$ is a finite sum of independent Bernoulli random variables. Then
\begin{equation}\label{chernoff}
P\left(X\leq \frac{E[X]}{2} \right)\leq \exp\left(-\frac{E[X]}{8} \right).
\end{equation}

\subsection{Heuristic}

We now give a heuristic for Theorem \ref{mainthm}. We hope that it will provide useful intuition for our proof.

Fix a vertex $y$ in $\partial B(r)$ for large $r$. In order to hit $y$ by time close to $r$, some particles must make long ``runs'' of consecutive steps towards $y$ without taking any steps away from $y$. (We associate each particle with all its ancestors, so that although technically a particle only lives for one unit of time, when we talk about it making a run of length $\ell$ towards $y$, we mean that it and its last $\ell-1$ ancestors all stepped towards $y$.)

Start from one particle at $0$, and let $Z_i$ be the number of particles at time $i$ that have taken $i$ steps towards $y$. Then $(Z_i, i\ge0)$ forms a critical Galton-Watson process. Although this process will eventually die out (likely before any particle hits $y$), its total progeny has infinite mean. This gives rise to a potentially large number of particles that have taken exactly \emph{one} step away from $y$. Suppose this number is $A$. Each of these $A$ particles starts another critical Galton-Watson tree of particles that have taken $i-1$ steps towards $y$ at time $i$. It is known that if we run $A$ independent critical Galton-Watson processes, or equivalently one critical Galton-Watson process starting with $A$ initial particles, then with probability of order $1$ it will survive for of order $A$ generations, with a total progeny of order $A^2$. This gives rise to of order $A^2$ particles that have taken exactly \emph{two} steps away from $y$. Repeating this argument $k$ times suggests that we should expect (very roughly) $A^{2^{k-1}}$ particles that have taken exactly $k$ steps away from $y$, giving rise to another critical Galton-Watson process starting with (very roughly) $A^{2^{k-1}}$ initial particles, which survives for (very roughly) $A^{2^{k-1}}$ generations.

As soon as one of these processes---say the $k$th---survives for $d(0,y)$ generations, then $y$ must have been hit by a particle, which will have taken $k$ steps in the wrong direction, and therefore $d(0,y)+2k$ steps in total. In other words, if $A^{2^{k-1}} > d(0,y)$ then $H(y)\le d(0,y)+2k$. The converse is not quite true, but the fact that $A^{2^{k-1}}$ grows so quickly means that it is \emph{almost} true, in that the correction is of smaller order. This implies that $H(y) \approx d(0,y) + (2/\log 2)\log\log d(0,y)$, which can be made rigorous and holds with high probability for each $y$.

We might thus expect the cover time of the ball of radius $r$ to equal roughly $r + (2/\log 2)\log\log r$. Indeed, this is essentially the explanation for the $(1/\log2)\log\log n$ term in \eqref{bramresult} (the extra factor of $2$ accounts for the fact that our particles cannot stay still, but must either move towards or away from $y$ at each step), and gives the correct answer to the cover time problem when $d=2$. The fact that we instead see $r+\frac{2}{\log(3/2)}\log\log r$ when $d\ge 3$ boils down to the fact that while \emph{most} vertices $y\in \partial B(r)$ are hit by time $r+(2/\log 2)\log\log r$, there are many vertices in $\partial B(r)$, and some are not hit until later.

To see how this happens, again suppose that we have $A$ particles that have taken exactly one step away from $y$. The probability that the resulting critical Galton-Watson process starting with $A$ initial particles survives for fewer than $A^{1/2+\eps}$ generations (for some small $\eps>0$) is roughly $(1-c/A^{1/2+\eps})^A \approx \exp(-cA^{1/2-\eps})$. In doing so the particles cover distance roughly $A^{1/2+\eps}$, and therefore the number of possible vertices $y$ at this distance from the origin that could see such behaviour is of order $(d-1)^{A^{1/2+\eps}}$. Since, for $A$ large and $d>3$, we have $(d-1)^{A^{1/2+\eps}} \gg \exp(cA^{1/2-\eps})$, we might expect that some vertices \emph{do} see such behaviour. The total number of particles seen if this occurs is of order $A\cdot A^{1/2+\eps} = A^{3/2+\eps}$.

Continuing recursively, we might expect that some vertices $y$ see only $A^{(3/2+\eps)^k}$ particles that have taken $k$ steps away from $y$. Following the same argument as above, we deduce that these vertices have hitting times satisfying
\[H(y) \approx d(0,y) + \frac{2}{\log (3/2+\eps)}\log\log d(0,y),\]
and since $\eps>0$ was arbitrarily small, this agrees with our desired result.

This argument gives us essentially a first moment estimate on the number of vertices whose hitting times are of the order stated in Theorem \ref{mainthm}. Unfortunately a naive second moment bound does not work, since the processes seen from two different vertices in the tree are highly dependent, especially if the two vertices are near each other. To get around this we use the tree structure of the graph strongly. We fix $r'<r$ and show that many vertices in $\partial B(r')$ behave ``normally'', in that they are not hit too early and the number of particles moving towards them is not too large. For each of these ``normal'' vertices $x$ we fix a vertex $z(x)\in \partial B(r)$ that is in the subtree rooted at $x$, by which we mean that any path from $0$ to $z$ must pass through $x$. Using the argument above, we estimate the probability that $z(x)$ is hit later than usual and has a relatively small number of particles moving towards it (given that $x$ is normal), and use the tree structure to get independence of these events for different vertices $x$. In fact, rather than just carrying out this procedure for the desired choice of $r$, we carry out a multi-scale argument using the scales dictated by the argument above: very roughly, $r_k\approx A^{(1/2+\eps)(3/2+\eps)^k}$ for each $k$.

\subsection{Open questions}

One open question is whether our main result, Theorem \ref{mainthm}, can be strengthened further, perhaps along the lines of the result \eqref{bramresult} of Bramson \cite{bramson:minimal_displacement_BRW}. It would also be interesting to give results when the offspring distribution has mean $m\in(\frac{d}{2\sqrt{d-1}},d)$. (When $m\le \frac{d}{2\sqrt{d-1}}$ some vertices remain uncovered for all time, and when $m>d$ it is easy to see that $\tcov(r) = r-O(1)$ almost surely for any $r$.).

Another option is to extend our results to other trees. For example, what is the cover time when $G$ is itself a non-trivial Galton-Watson tree? The speed of simple random walk on (non-trivial) Galton-Watson trees with mean offspring distribution $d-1$ is slower than on regular $d$-ary trees, and our proof techniques no longer apply. Work is underway to at least partially address this question.

Further, one may ask for the cover time of tree-indexed random walk on trees, where the time tree has branching number $d$, and the space tree has branching number $d'$. See \cite{benjamini_peres:mcs_trees} for the study of tree-indexed random walks and the book \cite{lyons_peres:probability_on_trees} for background. Variants of BRW and tree-indexed random walks were used in \cite{benjamini_schramm:cheeger_trees, sudakov_vondrak:embeddings_trees} to study the embedding of trees into graphs.

\section{Freezing particles}\label{freezing_sec}

\subsection{Freezing particles after one step in the wrong direction}

Fix a vertex $x$ in our $d$-ary tree $T$. Consider our usual BRW (with offspring mean $d$) on $T$, but freeze any particle (that is, prevent it from moving or branching) as soon as it either (a) takes a step away from $x$, or (b) reaches $x$, whichever happens first. In this picture, let $Y_x$ be the number of particles that hit $x$, let $F_x$ be the number of particles that are frozen as they step away from $x$, and let $S_x$ be the total number of particles ever seen (until the time that all particles have become frozen). If we start with any finite collection of particles, then $S_x$ is finite since at each step any non-frozen particle must step either towards or away from $x$ and particles are frozen as soon as they step away from $x$ or reach $x$.

For $x\in T$ and $r\ge 0$, let $T(x,r)$ be the set of vertices at distance $\lfloor r \rfloor$ from $x$ in the subtree rooted at $x$; that is, those vertices at distance $\lfloor r\rfloor$ from $x$ and $\lfloor r\rfloor +d(0,x)$ from $0$. We aim to provide upper and lower bounds on the number of particles in the freezing process outlined above. First we give an upper bound in the form of the following simple expectation calculation.

\begin{lem}\label{nopartspt1}
Fix $n\in \N$, $R\ge 1$ and $x\in T$. Suppose that $\Gamma$ consists of vertices which are at distance at most $R$ from $x$. Then
\[\E_\Gamma[Y_x] = |\Gamma| \,\,\text{ and }\,\, \E_\Gamma[S_x] \le (d+1)R|\Gamma|.\]
\end{lem}

\begin{proof}
Fix a vertex $v\in\Gamma$ and label the vertices in the path from $v$ to $x$ as $v=v_0, v_1,\ldots, v_k=x$. Let $Z_j$ be the number of particles that reach $v_j$ after $j$ steps, for each $j\le k$. For $j< k$, each particle at $v_j$ independently has a random number of children with distribution $\mu$, which has mean $d$ and finite variance, each of which moves to $v_{j+1}$ with probability $1/d$. Thus the sequence $(Z_j, j=0,\ldots,k)$ forms a critical Galton-Watson process with finite variance, stopped at generation $k$. As a result,
\[\E_{1,v}[Y_x] = \E_{1,v}[Z_k] = 1.\]
Now, the total number of particles seen is exactly those that contribute to $Z_j$ for $0\le j\le k-1$, together with their frozen children. Thus
\[\E_{1,v}[S_x] \le (d+1) \E_{1,v}\Big[\sum_{j=0}^{k-1} Z_j\Big] = (d+1)k\]
and since $v\in\Gamma$ we have $k=d(v,x)\le R$ so $\E_{1,v}[S_x]\le (d+1)R$. To complete the proof of the lemma we simply sum over $v\in\Gamma$.
\end{proof}

For a lower bound it is easier to bound the number of unfrozen particles, rather than the number of frozen particles. However, we will eventually need to bound the number of frozen particles, so we will need the following lemma which checks that if the number of unfrozen particles is large then the number of frozen particles should be large.

\begin{lem}\label{StoF}
There exists a constant $\nu\in(0,1/2)$ such that for any $R,M\in\N$ and $v,x\in T$ with $d(v,x)\le R$,
\[\P_{1,v}(F_x + Y_x \le \nu S_x - RM) \le Re^{-\nu M}.\]
\end{lem}

\begin{proof}
As in the proof of Lemma \ref{nopartspt1}, label the vertices in the path from $v$ to $x$ as $v=v_0, v_1,\ldots, v_R=x$. For each $j\le R-1$, let $Z_j$ be the number of (non-frozen) particles that reach $v_j$ after $j$ steps, and $W_j$ be the number of these particles that have at least one child that is frozen as it stepps away from $x$.

Note that for each non-frozen particle at $v_j$, the event that it has \emph{no} children that step away from $x$ is independent of other non-frozen particles at $v_j$, and has probability
\[\mu(0) + \mu(1)(1/d) + \mu(2)(1/d)^2 + \ldots = \E[1/d^L] < 1\]
where $L$ is a random variable with distribution $\mu$. Therefore for each $j\le R-1$,
\[\E_{1,v}[W_j|Z_j] = Z_j(1-\E[1/d^L]).\]
Given the value of $Z_j$, $W_j$ is the sum of $Z_j$ indepdendent Bernoulli random variables. Thus we can apply the Chernoff bound \eqref{chernoff}, giving
\[\P_{1,v}\big(W_j \le \E_{1,v}[W_j|Z_j]/2 \,\big|\, Z_j\big) \le e^{-\E_{1,v}[W_j|Z_j]/8}\]
Setting $\kappa= 1-\E[1/d^L] \in(0,1)$ and combining the two equations above, we have
\[\P_{1,v}(W_j \le \kappa Z_j/2 \,|\, Z_j) \le e^{-\kappa Z_j/8}\]
and so
\[\P_{1,v}(Z_j \ge M \text{ and } W_j \le \kappa Z_j/2 ) \le e^{-\kappa M/8}.\]
By a union bound,
\begin{equation}\label{unionbound}
\P_{1,v}(\exists j\le R-1 : Z_j \ge M \text{ and } W_j \le \kappa Z_j/2 ) \le Re^{-\kappa M/8}.
\end{equation}

Since $S_x = \sum_{j=0}^{R-1} Z_j + F_x + Y_x$, we have
\begin{align*}
\P_{1,v}(F_x + Y_x \le \nu S_x - R M) &= \P_{1,v}\Big(F_x + Y_x \le \nu\sum_{j=0}^{R-1}Z_j + \nu F_x +\nu Y_x - R M\Big)\\
&= \P_{1,v}\Big( (1-\nu)(F_x+Y_x) \le \nu \sum_{j=0}^{R-1}Z_j - RM\Big)\\
&\le \P_{1,v}\Big( (1-\nu)F_x \le \nu \sum_{j=0}^{R-1}Z_j - (1-\nu)RM\Big)
\end{align*}
where the last line follows trivially since $Y_x\ge 0$ and $RM\ge 0$. Now note that $F_x \ge \sum_{j=0}^{R-1} W_j$, so following on from the above,
\begin{align*}
\P_{1,v}(F_x+Y_x \le \nu S_x - R M) &\le \P_{1,v}\Big( (1-\nu)\sum_{j=0}^{R-1} W_j \le \nu \sum_{j=0}^{R-1}Z_j - (1-\nu)RM\Big)\\
&= \P_{1,v}\Big( \sum_{j=0}^{R-1} (W_j+M) \le \frac{\nu}{1-\nu} \sum_{j=0}^{R-1}Z_j \Big)\\
&\le \P_{1,v}\Big(\exists j\le R-1 : W_j + M \le \frac{\nu}{1-\nu} Z_j \Big). 
\end{align*}
Choosing $\nu\in(0,1/2)$ such that $\nu/(1-\nu)\le \kappa/2$ and $\nu\le\kappa/8$, the result follows from \eqref{unionbound}.
\end{proof}

We can now give our lower bound on the number of frozen particles.

\begin{lem}\label{existparts}
Suppose that $A,n\in\N$, $x\in T$ and $y\in T(x,A)$. Suppose also that $\Gamma$ consists of at least $n$ vertices, none of which is in the subtree rooted at $x$ except possibly at $x$ itself, and all of which are at distance at most $2A$ from $y$. Then provided that $A$ is sufficiently large, we have
\[\P_\Gamma( F_y + Y_y \le \delta A n ) \le e^{-\delta n/A}\]
for some constant $\delta\in(0,1]$ depending only on $d$ and the variance of $\mu$.
\end{lem}

\begin{proof}[Proof of Lemma \ref{existparts}]
Without loss of generality we may assume that $|\Gamma|=n$. Label the initial $n$ particles from $1$ to $n$ and say that particle $i$ starts from vertex $v_i$. Run the $y$-freezing process and set $X_i$ equal to $1$ if particle $i$ has at least $A^2$ descendants in total, and $X_i=0$ otherwise. Fix $\nu>0$ as in Lemma \ref{StoF} and let $X'_i$ equal $1$ if particle $i$ has at least $\nu A^2 / 2$ frozen descendants (that is, descendants that contribute to either $F_y$ or $Y_y$), and $X'_i=0$ otherwise.

By the argument in the proof of Lemma \ref{nopartspt1}, the number of non-frozen descendants of particle $i$ after $0,1,2,\ldots,d(v_i,y)-1$ steps forms a critical Galton-Watson process with finite variance, stopped at generation $d(v_i,y)-1$. Thus the total number of non-frozen descendants of particle $i$ is distributed as the total progeny up to generation $d(v_i,y)-1 \ge A-1$ of a Galton-Watson process with mean offspring number $1$ and finite variance. By \eqref{totalprogeq}, the probability that a critical Galton-Watson process with finite variance has total progeny up to generation $A-1$ of at least $A^2$ is at least $c/A$, for some constant $c>0$ depending on the variance. Thus, for each $i$,
\[\P_\Gamma(X_i = 1) \ge c/A.\]
Also, by Lemma \ref{StoF}, choosing $M=\nu A/4$ and $R=2A$,
\[\P_\Gamma(X_i=1 \text{ but } X'_i=0) \le 2Ae^{-\nu^2 A/4}.\]
which for $A$ sufficiently large is at most $c/(2A)$. Therefore, for $A$ sufficiently large,
\[\P_\Gamma(X'_i = 1) \ge \P_\Gamma(X_i = 1) - \P_\Gamma(X=1 \text{ but } X'=0) \ge \frac{c}{A} - \frac{c}{2A} = \frac{c}{2A}.\]

Letting $X=\sum_{i=1}^n X'_i$, we have
\[\E_\Gamma[X] \ge \frac{cn}{2A}.\]
Since $X$ is a sum of independent Bernoulli random variables we can apply the Chernoff bound \eqref{chernoff}, obtaining
\[\P_\Gamma\Big(X\le \frac{cn}{4A}\Big) \le \exp\left(-\frac{cn}{16A}\right).\]
But if $X>cn/(4A)$, then $F_y+Y_y$ must be at least
\[\frac{cn}{4A}\cdot \frac{\nu A^2}{2} = \frac{c\nu n A}{8}.\]
Since $\nu<1/2$, choosing $\delta = \min\{c\nu/8,1\}$ gives the result.
\end{proof}

\subsection{Freezing particles after $k$ steps in the wrong direction}

Fix a vertex $x\in T$ and $k\in\N$, and consider our original branching random walk, but this time freeze any particle (that is, prevent it from moving or branching) as soon as it either (a) takes its $k$th step away from $x$, or (b) hits $x$, whichever happens first. Let $Y^{(k)}_x$ be the number of particles that are frozen at $x$ in this picture, and $F^{(k)}_x$ be the number of particles that are frozen as they take their $k$th step away from $x$. Since this picture depends on the choice of $x$ and $k$, sometimes we may call these particles $(x,k)$-frozen. We also let $S^{(k)}_x$ be the total number of particles ever seen in this picture. Note that $Y^{(1)}_x = Y_x$, $F^{(1)}_x = F_x$ and $S^{(1)}_x = S_x$. Let $\F_x^k$ be the $\sigma$-algebra generated by the $(x,k)$-freezing process.

We will use this freezing procedure in both the lower and upper bounds for Theorem \ref{mainthm}. In the remainder of this section we aim to prove the following proposition, which uses Lemma \ref{nopartspt1} and will be used in the lower bound for Theorem \ref{mainthm}.

\begin{prop}\label{nopartsprop}
Suppose that $k,n\in\N$, $k\le n^{1/2}$ and $x\in T(0,n^{1/2})$. Then on the event $\{F^{(k-1)}_x \le n \hbox{ and } Y^{(k-1)}_x=0\}$ we have
\[\P\big(\exists z\in T(x, A n^{1/2}) : Y^{(k)}_z = 0 \,\hbox{ and }\, F^{(k)}_z \le A n^{3/2} \,\big|\, \F^{k-1}_x \big) \ge 1- c/A^{1/2}\]
for some constant $c$ (depending only on $d$ and the variance of $\mu$) and all large $A$ and $n$.
\end{prop}

We now aim to prove this result. Take $a\ge 5d$. We will use a two-stage argument, first showing that there are, with high probability, many vertices $y$ in $T(x, a n^{1/2})$ that satsify $Y_y^{(k)} < a n$ and $F_y^{(k)} < 3a^2 n^{3/2}$. We call such vertices ``good''. Then in the second stage we will show that with high probability, there is a vertex $z$ in $T(y,an^{1/2})$ with our desired properties for at least one of the good vertices $y$.

To make this argument rigorous, define
\[M_n(x) = \{y\in T(x,a n^{1/2}) : Y^{(k)}_y < a n \,\,\hbox{ and }\,\, F^{(k)}_y < 3a^2 n^{3/2}\}\]
and
\[\bar M_n(x) = \{y\in T(x,a n^{1/2}) : Y^{(k)}_y \ge a n \,\,\hbox{ or }\,\, F^{(k)}_y \ge 3a^2 n^{3/2}\}.\]
That is, $M_n(x)$ is the set of good vertices and $\bar M_n(x)$ is its complement in $T(x,a n^{1/2})$.

Our first aim is to show that $M_n(x)$ is large with high probability, on the event that $F^{(k-1)}_x\le n$ and $Y^{(k-1)}_x=0$.

\begin{lem}\label{Mnlarge}
Suppose that $k,n\in\N$, $k\le n^{1/2}$, $a\ge 4d$ and $x\in T(0,n^{1/2})$. Then on the event $\{F^{(k-1)}_x \le n \text{ and } Y^{(k-1)}_x=0\}$ we have
\[\P\Big( |M_n(x)| \ge \frac{1}{2} |T(x,a n^{1/2})| \,\Big|\, \F_x^{k-1}\Big) \ge 1-\frac{4d}{a}.\]
\end{lem}

\begin{proof}
Fix $y\in T(x,a n^{1/2})$. Label the locations of the $(x,k-1)$-frozen particles as $v_1,\ldots,v_m$. Since these particles have taken at most $k-1$ steps away from $x$, we know that $d(v_i,y)\le k+d(0,x)+d(x,y) \le (1+1+a)n^{1/2} \le 3an^{1/2}$. Let $\Gamma=(x_1,\ldots,x_m)$. By Lemma \ref{nopartspt1}, we have
\[\E_\Gamma[Y^{(1)}_y] = \E_\Gamma[Y_y] = m \,\,\text{ and }\,\, \E_\Gamma[F^{(1)}_y] \le \E_\Gamma[S_y] \le (d+1)\cdot 3an^{1/2}\cdot m.\]
Further note that if $Y^{(k-1)}_x=0$ then running the $(y,1)$-freezing process from the starting configuration consisting of the $(x,k-1)$-frozen particles is equivalent to running the $(y,k)$-freezing process. Thus, applying Markov's inequality, if $Y^{(k-1)}_x = 0$, then
\begin{align*}
\P\big(Y^{(k)}_y \ge an \,\,\text{ or }\,\, F^{(k)}_y \ge 3a^2n^{3/2} \,\big|\, \F_x^{k-1}\big) &\le \P_\Gamma(Y^{(1)}_y \ge an \,\,\text{ or }\,\, F^{(1)}_y \ge 3a^2n^{3/2})\\
&\le \frac{\E_\Gamma[Y^{(1)}_y]}{an} + \frac{\E_\Gamma[F^{(1)}_y]}{3a^2 n^{3/2}}\\
&\le \frac{m}{an} + \frac{3a(d+1)mn^{1/2}}{3a^2 n^{3/2}}.
\end{align*}
Note that on the event $\{F^{(k-1)}_x \le n \text{ and } Y^{(k-1)}_x=0\}$ we have $|\Gamma|=m\le n$, so by the above, on this event,
\[\P\big(Y^{(k)}_y \ge an \,\,\text{ or }\,\, F^{(k)}_y \ge 3a^2n^{3/2} \,\big|\, \F_x^{k-1}\big) \le (2+d)/a \le 2d/a.\]

Then, again on the event $\{F^{(k-1)}_x \le n \text{ and } Y^{(k-1)}_x=0\}$,
\[\E\big[|\bar M_n(x)|\,\big|\, \F_x^{k-1}\big] = |T(x,a n^{1/2})|\cdot \P\big( Y^{(k)}_y \ge a n \,\,\hbox{ or }\,\, F^{(k)}_y \ge 3a^2 n^{3/2} \,\big|\, \F_x^{k-1} \big) \le \frac{2d}{a} |T(x,a n^{1/2})|.\]
Thus, applying Markov's inequality again, on the event $\{F^{(k-1)}_x \le n \text{ and } Y^{(k-1)}_x=0\}$ we have
\[\P\Big(|\bar M_n(x)| \ge \frac{1}{2} |T(x,a n^{1/2})| \,\Big|\, \F_x^{k-1}\Big) \le \frac{4d}{a}\]
and therefore, since $|M_n(x)| + |\bar M_n(x)| = |T(x,a n^{1/2})|$, on the event $\{F^{(k-1)}_x \le n \text{ and } Y^{(k-1)}_x=0\}$ we have
\[\P\Big( |M_n(x)| \ge \frac{1}{2} |T(x,a n^{1/2})| \,\Big|\, \F_x^{k-1}\Big) \ge 1-\frac{4d}{a},\]
as required.
\end{proof}

Next we aim to bound from below the probability that, if $y$ is a good vertex, then there is a vertex $z$ in $T(y,a n^{1/2})$ such that $Y_z=0$.

\begin{lem}\label{xprimeprob}
Suppose that $k,n\in\N$, $k\le n^{1/2}$, $a\ge 4d$ and $x\in T(0,n^{1/2})$. There exists $c>0$ such that if $y\in M_n(x)$ and $n$ is large, then for any $z\in T(y,an^{1/2})$ we have
\[\P\big(Y^{(k)}_{z(y)} = 0 \hbox{ and } F^{(k)}_{z(y)}-F^{(k)}_{y} \le 2(d+1) a^2 n^{3/2} \,\big|\, \F_y^k\,\big) \ge \frac12 e^{-2 c n^{1/2}}.\]
\end{lem}

\begin{proof}
Note that if we start with one particle at $0$, then in order for a particle to be $(z(y),k)$-frozen, it must also be $(y,k)$-frozen. In fact, to contribute to either $Y^{(k)}_{z(y)}$ or $F^{(k)}_{z(y)}-F^{(k)}_y$, a particle must be $(y,k)$-frozen at $y$ specifically. Thus, if $y\in M_n(x)$, then
\[\P\big(Y^{(k)}_{z(y)} = 0 \hbox{ and } F^{(k)}_{z(y)}-F^{(k)}_{y} \le 2(d+1) a^2 n^{3/2} \,\big| \F_y^k\big) \ge \P_{\lfloor a n\rfloor,y}(Y_{z(y)} = 0 \hbox{ and } F_{z(y)} \le 2(d+1) a^2 n^{3/2}).\]
Of course $F_{z(y)} \le S_{z(y)}$, so for $y\in M_n(x)$ we have
\[\P\big(Y^{(k)}_{z(y)} = 0 \hbox{ and } F^{(k)}_{z(y)}-F^{(k)}_{y} \le 2(d+1) a^2 n^{3/2} \,\big| \F_y^k\big) \ge \P_{\lfloor a n\rfloor,y}(Y_{z(y)} = 0 \hbox{ and } S_{z(y)} \le 2(d+1) a^2 n^{3/2}).\]
The events $\{Y_{z(y)} = 0\}$ and $\{S_{z(y)} \le 2(d+1) a^2 n^{3/2})\}$ are both decreasing (on the set of finite trees with the partial order $t\le t'$ if $t$ is a subtree of $t'$) so by the FKG inequality \cite{FKG},
%The set of finite trees with this partial order is a finite distributive lattice using inclusion and exclusion. It satisfies the FKG lattice condition with equality since you can write P(x) as a product over the vertices in x of the probability that that vertex has the required number of children. Then the collection of vertex offspring sizes of x and y equals the collection of vertex offspring sizes of (x\vee y) and (x\wedge y), so P(x)P(y) = P(x\vee y)P(x\wedge y).
\[\P_{\lfloor a n\rfloor,y}(Y_{z(y)} = 0 \hbox{ and } S_{z(y)} \le 2(d+1) a^2 n^{3/2}) \ge \P_{\lfloor a n\rfloor,y}(Y_{z(y)} = 0) \P_{\lfloor a n\rfloor,y}(S_{z(y)} \le 2(d+1) a^2 n^{3/2})\]
and thus, for $y\in M_n(x)$,
\begin{multline}\label{vxand}
\P\big(Y^{(k)}_{z(y)} = 0 \hbox{ and } F^{(k)}_{z(y)}-F^{(k)}_{y} \le 2(d+1) a^2 n^{3/2} \,\big|\, \F_y^k\big) \\
\ge \P_{\lfloor a n\rfloor,y}(Y_{z(y)} = 0) \P_{\lfloor a n\rfloor,y}(S_{z(y)} \le 2(d+1) a^2 n^{3/2}).
\end{multline}

Note that, starting with any number $j\in\N$ of particles at $y$,
\[\P_{j,y}(Y_{z(y)} = 0) = \P_{1,y}(Y_{z(y)}=0)^j.\]
Recall from the proof of Lemma \ref{existparts} that under $\P_{1,y}$, the event that $Y_{z(y)}$ equals zero is the event that a critical Galton-Watson tree (with finite variance) survives for fewer than $\lfloor a n^{1/2}\rfloor$ generations; this has probability at least $1-c/(a n^{1/2})$ for some finite constant $c$ by Lemma \ref{kollem}. Thus for large $n$
\[\P_{j,y}(Y_{z(y)} = 0) \ge \Big(1-\frac{c}{a n^{1/2}}\Big)^j \ge \exp\Big(-\frac{2cj}{a n^{1/2}}\Big)\]
where for the second inequality we used the fact that $1-u \ge e^{-2u}$ for $0\le u\le 1/2$.
Also
\[\P_{j,y}(S_{z(y)} \le 2(d+1) a^2 n^{3/2}) = 1-\P_{j,y}(S_{z(y)} > 2(d+1) a^2 n^{3/2}) \ge 1-\frac{\E_{j,y}[S_{z(y)}]}{2(d+1) a^2 n^{3/2}}\]
so applying Lemma \ref{nopartspt1},
\[\P_{j,y}(S_{z(y)} \le 2(d+1) a^2 n^{3/2}) \ge 1-\frac{(d+1)\cdot an^{1/2} \cdot j}{2(d+1) a^2 n^{3/2}} = 1 - \frac{j}{2an}.\]
Substituting these bounds back into \eqref{vxand}, we get that for $y\in M_n(x)$,
\[\P\big(Y^{(k)}_{z(y)} = 0 \hbox{ and } F^{(k)}_{z(y)}-F^{(k)}_{y} \le 2(d+1) a^2 n^{3/2} \,\big|\, \F_y^k\,\big) \ge \exp\Big(-\frac{2c an}{a n^{1/2}}\Big) \cdot \Big(1 - \frac{an}{2an}\Big) = \frac12 e^{-2 c n^{1/2}}\]
which completes the proof.
\end{proof}

Now we prove Proposition \ref{nopartsprop} by putting the estimates from Lemmas \ref{Mnlarge} and \ref{xprimeprob} together.

\begin{proof}[Proof of Proposition \ref{nopartsprop}]
Take $a\ge 5d$. For each $y\in T(x,a n^{1/2})$, arbitrarily choose a vertex $z(y)\in T(y,an^{1/2})$. Say that $y\in T(x,an^{1/2})$ is ``special'' if $Y^{(k)}_{z(y)} = 0$ and $F^{(k)}_{z(y)}-F^{(k)}_{y} \le 2(d+1) a^2 n^{3/2}$. Let 
\[\G_x^{k,r} = \sigma\bigg(\bigcup_{y\in T(x,r)} \F_y^k\bigg),\]
the $\sigma$-algebra generated by the $(y,k)$-freezing processes for all $y\in T(x,r)$. We note first that if $y$ is both good and special, then $Y^{(k)}_{z(y)} = 0$ and $F^{(k)}_{z(y)} \le (2d+5)a^2 n^{3/2}$. Thus, using also that $\F_x^{k-1} \subset \G_x^{k,r}$ for any $r\ge0$,
\begin{align*}
&\P\big(\exists z\in T(x,2an^{1/2}) : Y^{(k)}_z = 0 \text{ and } F^{(k)}_z \le (2d+5)a^2 n^{3/2} \,\big|\, \F_x^{k-1} \big)\\
&\hspace{20mm}\ge \P\Big( |M_n(x)| \ge \frac12 |T(x,an^{1/2})|,\,\exists y\in M_n(x) : y \text{ is special} \,\Big|\, \F_x^{k-1} \Big)\\
&\hspace{20mm}=\E\Big[ \ind_{\{|M_n(x)| \ge \frac12 |T(x,an^{1/2})|\}} \P\Big(\exists y\in M_n(x) : y \text{ is special} \,\Big|\, \G_x^{k,\lfloor an^{1/2}\rfloor} \Big) \,\Big|\, \F_x^{k-1}\Big].
\end{align*}
Given $\G_x^{k,\lfloor an^{1/2}\rfloor}$, the events $\big\{\{y \text{ is special}\} : y\in T(x,an^{1/2})\big\}$ are independent. Therefore the above is at least
\[\E\bigg[ \ind_{\{|M_n(x)| \ge \frac12 |T(x,an^{1/2})|\}}\bigg(1- \prod_{y\in M_n(x)}\P\big(y \text{ is not special} \,\big|\, \G_x^{k,\lfloor an^{1/2}\rfloor} \big)\bigg) \,\bigg|\, \F_x^{k-1}\bigg].\]
For any $y\in T(x,an^{1/2})$, the event that $y$ is special depends on $\G_x^{k,\lfloor an^{1/2}\rfloor}$ only through the value of $Y^{(k)}_y$; thus
\[\P\big(y \text{ is not special} \,\big|\, \G_x^{k,\lfloor an^{1/2}\rfloor} \big) = \P\big(y \text{ is not special} \,\big|\, \F_y^k \big).\]
By Lemma \ref{xprimeprob}, if additionally $y\in M_n(x)$, we have
\[\P\big(y \text{ is not special} \,\big|\, \F_y^k\big) \le 1-\frac{1}{2} \exp(-2c n^{1/2}) \le \exp(- e^{-2c n^{1/2}}/2),\]
and putting all this together we have shown that
\begin{align*}
&\P\big(\exists z\in T(x,2an^{1/2}) : Y^{(k)}_z = 0 \text{ and } F^{(k)}_z \le (2d+5)a^2 n^{3/2} \,\big|\, \F_x^{k-1} \big)\\
&\hspace{20mm}\ge \E\bigg[ \ind_{\{|M_n(x)| \ge \frac12 |T(x,an^{1/2})|\}}\bigg(1- \prod_{y\in M_n(x)}\exp(- e^{-2c n^{1/2}}/2)\bigg) \,\bigg|\, \F_x^{k-1}\bigg]\\
&\hspace{20mm}\ge \P\Big( |M_n(x)| \ge \frac12 |T(x,an^{1/2})| \,\Big|\, \F_x^{k-1}\Big) \Big(1- \exp(- |T(x,an^{1/2})|e^{-2c n^{1/2}}/4)\Big)
\end{align*}
By Lemma \ref{Mnlarge}, on the event $\{F^{(k-1)}_x \le n \text{ and } Y^{(k-1)}_x=0\}$ this is at least
\[\Big(1-\frac{4d}{a}\Big)\Big(1-\exp\big(- e^{-2c n^{1/2}}|T(x,an^{1/2})|/4\big)\Big).\]

Since $|T(y,a n^{1/2})| \ge (d-1)^{a n^{1/2}}$, we have $\exp\big(-e^{-2cn^{1/2}}|T(y,a n^{1/2})|/4\big) \le 1/a$ provided that $a$ and $n$ are large, so the above is at least $1-5d/a$. And of course if $Y^{(k)}_z=0$ and $F^{(k)}_z\le (2d+5)a^2 n^{3/2}$, then for any $j\ge0$, any vertex $v\in T(z,j) \subset T(x, 2an^{1/2}+j)$ also has $Y^{(k)}_v=0$ and $F^{(k)}_v\le (2d+5)a^2 n^{3/2}$. Thus
\[\P\big(\exists v\in T(x,(2d+5)a^2 n^{1/2}) : Y^{(k)}_v = 0 \text{ and } F^{(k)}_v \le (2d+5)a^2 n^{3/2} \,\big|\, \F_x^{k-1} \big) \ge 1-5d/a.\]
Writing $A=(2d+5)a^2$ completes the proof.
\end{proof}

\section{Proof of the lower bound in Theorem \ref{mainthm}}\label{lbsec}

We aim to prove that for any $\eta>0$,
\[\P\Big(\liminf_{n\to\infty} \frac{\tcov(n)-n}{\log\log n} < \frac{2}{\log(3/2)}-\eta\Big) = 0.\]
Take $\delta\in(0,1/4)$ small and $M$ large, both to be fixed later. 
For $k\ge1$, let
\[n_k = M^{(3/2+\delta)^k}, \,\,\,\, p_k = M^{-\delta(3/2+\delta)^{k-1}/2}, \,\,\,\,\hbox{ and }\,\,\,\, R_k = \sum_{j=0}^{k-1} M^{(1/2+\delta)(3/2+\delta)^j}.\]
Say that $x\in \partial B(R_{k})$ is \emph{slow} if $Y_x^{(k)}=0$ and $F_x^{(k)}\le n_k$. Define $\mathcal A_k$ to be the event that there is at least one slow vertex in $\partial B(R_k)$; that is,
\[\mathcal A_k = \{\exists z\in \partial B(R_k) : Y^{(k)}_z = 0 \hbox{ and } F^{(k)}_z \le n_k\}.\]
Let $X_k$ be a uniformly chosen slow vertex in $\partial B(R_k)$, or $X_k=0$ if there are no such vertices.

Note that when $M$ is large, $R_{k-1}\le n_{k-1}^{1/2}$ and $k\le n_k^{1/2}$. Thus, setting $A=n_{k-1}^\delta$ so that $n_k = An_{k-1}^{3/2}$, $R_k-R_{k-1}= An_{k-1}^{1/2}$ and $p_k=A^{-1/2}$, Proposition \ref{nopartsprop} tells us that for any slow $x\in \partial B(R_{k-1})$, we have
\[\P\big(\exists z\in T(x,R_{k}-R_{k-1}) : Y^{(k)}_z = 0 \,\hbox{ and }\, F^{(k)}_z \le n_{k} \,\big|\, \F^{k-1}_x \big) \ge 1- cp_k.\]
In particular,
\[\P(\mathcal A_k \cap \mathcal A_{k-1}) = \E[\P(\mathcal A_k \,|\, \F^{k-1}_{X_{k-1}}) \ind_{\mathcal A_{k-1}}] \ge (1-cp_k)\P(\mathcal A_{k-1})\]
and therefore
\[\P(\mathcal A_k^c \cap \mathcal A_{k-1}) \le cp_k\P(\mathcal A_{k-1})\le cp_k.\]

Since
\[\P\Big(\bigcup_{j=1}^k \mathcal A_j^c \Big) \le \P\Big(\bigcup_{j=1}^{k-1} \mathcal A_j^c \Big) + \P(\mathcal A_k^c, \mathcal A_{k-1}) \le \P\Big(\bigcup_{j=1}^{k-1} \mathcal A_j^c \Big) + cp_k,\]
by induction we have
\[\P\Big(\bigcup_{j=1}^k \mathcal A_j^c \Big) \le c\sum_{j=2}^k p_j + \P(\mathcal A_1^c) \le c\sum_{j=2}^\infty p_j + \P(\mathcal A_1^c).\]
Note that we can make $\P(\mathcal A_1^c)$ arbitrarily small by choosing $M$ sufficiently large, since for large enough $r$ \emph{any} vertex $z\in\partial B(r)$ satisfies $Y_z=0$ and $F_z\le r$ with probability at least $1-\eps$.

Choose $\eps>0$ and $\delta>0$ arbitrarily small, and $M$ large enough that $c\sum_{j=2}^\infty p_j + \P(\mathcal A_1^c) < \eps$. Then $\P(\bigcup_k \mathcal A_k^c)<\eps$. On the event $\mathcal A_k$, there is a vertex $z\in \partial B(R_k)$ such that no particles hit $z$ without first taking at least $k$ steps away from $z$. In this case the first hitting time of $z$ (and all its descendants in the tree) must be at least $R_k + 2k$. We deduce that
\[\P(\tcov(r) \ge r + 2k\,\,\forall r\ge R_k, \,\, \forall k\ge 1) > 1- \eps.\]
All that remains now is to invert $R_k$. Note that if $k\le \frac{\log\log n}{\log(3/2+3\delta)}$ then $\exp((3/2+3\delta)^k)\le n$, so if $n$ is large then $M^{(3/2+2\delta)^k}\le n$ and indeed $R_k \le n$. Therefore
\[\P\Big(\tcov(n) \ge n + \frac{2\log\log n}{\log(3/2+3\delta)} \,\, \hbox{ for all large } n\Big) > 1 - \eps,\]
and since $\delta$ and $\eps$ were arbitrary, this completes the proof of the lower bound in Theorem \ref{mainthm}.\qed

\section{Proof of the upper bound in Theorem \ref{mainthm}}\label{ubsec}

We want to show that for any $\eta>0$,
\[\P\Big(\limsup_{n\to\infty} \frac{\tcov(n)-n}{\log\log n} > \frac{2}{\log(3/2)}+\eta\Big) = 0.\]

Suppose that $x\in T$ and $y\in T(x,A)$, for some large $A\in\N$. Lemma \ref{existparts} tells us that if $\Gamma$ consists of at least $n$ particles none of which is in the subtree rooted at $x$ except possibly at $x$ itself, and none of which have distance greater than $2A$ from $y$, then
\[\P_\Gamma(F_y+Y_y \le \delta A n) \le e^{-\delta n/A}\]
where $\delta\in(0,1]$ is some fixed constant.

Now fix $N\in\N$. If we start with $N$ particles all at $0$, then (for any $k\ge 2$) none of the $(x,k-1)$-frozen particles are within the subtree rooted at $x$ except possibly at $x$ itself, and all have distance at most $d(0,y)+k$ from $y$. Thus if $d(0,y)+k\le 2A$, then recalling that $\F_x^{k-1}$ is the $\sigma$-algebra generated by the $(x,k-1)$-freezing process, on the event $\{F^{(k-1)}_x + Y^{(k-1)}_x \ge n\}$ we have
\[\P_{N,0}(F^{(k)}_y + Y^{(k)}_y \le \delta A n \,|\, \F^{k-1}_x ) \le e^{-\delta n/A}.\]

For each $k\in\N$ set
\[N_k = \frac{e^{(3/2)^k}}{\delta^2 a^2} \hspace{4mm} \hbox{ and } \hspace{4mm} R_k = \sum_{j=1}^{k-1} a N_j^{1/2}\]
for some small $a>0$ to be chosen later, and for $z\in T$ define the event
\[\mathcal B^k_z = \{F^{(k)}_z + Y^{(k)}_z \ge N_k\}.\]
Note that provided $a$ is sufficiently small, we have $R_k+k\le 2 aN_{k-1}^{1/2}$ for all $k\in\N$. Thus by the argument above, if $k\ge 2$, $x\in \partial B(R_{k-1})$ and $y\in T(x, a N_{k-1}^{1/2})$, %(this means that $d(0,y)=R_k$)
 then on the event $\mathcal B^{k-1}_x$ we have
\[\P_{N,0}\big(F^{(k)}_y + Y^{(k)}_y \le \delta a N_{k-1}^{3/2} \,\big|\, \F^{k-1}_x \big) \le \exp\big(-\delta N_{k-1} / aN_{k-1}^{1/2}\big).\]
Since $\delta a N_{k-1}^{3/2} = N_k$, we deduce that
\[\P_{N,0}\big((\mathcal B^k_y)^c \cap \mathcal B^{k-1}_x \big) \le \exp\big(-\delta N_{k-1}^{1/2} / a\big).\]

Let $\mathcal B_k = \bigcap_{x\in \partial B(R_k)} \mathcal B^k_x$. There are $d(d-1)^{R_k-1}$ vertices in $\partial B(R_k)$, so a union bound gives
\[\P_{N,0}\big(\mathcal B_k^c \cap \mathcal B_{k-1} \big) \le d(d-1)^{R_k-1}\exp\big(-\delta N_{k-1}^{1/2} / a\big).\]
Since
\[\P_{N,0}\Big(\bigcup_{j=1}^k \mathcal B_j^c \Big) \le \P_{N,0}\Big(\bigcup_{j=1}^{k-1} \mathcal B_j^c \Big) + \P_{N,0}(\mathcal B_k^c, \mathcal B_{k-1}) ,\]
by induction we have
\begin{equation}\label{Zunionupperbd}
\P_{N,0}\Big(\bigcup_{j=1}^k \mathcal B_j^c \Big) \le \sum_{j=2}^k d(d-1)^{R_j-1}\exp\big(-\delta N_{j-1}^{1/2} / a\big) + \P_{N,0}(\mathcal B_1^c).
\end{equation}

Now fix $k\ge1$ and suppose that $\mathcal B_k$ occurs, so for each $z\in \partial B(R_k)$, there are at least $N_k$ particles that are $(z,k)$-frozen. All such particles have taken at most $k$ steps away from $z$ when they are frozen. Each of these particles has distance at most $R_k+k$ from $z$, and therefore the probability that it has a descendant that hits $z$ without taking any more steps away from $z$ is bounded from below by the probability that a critical Galton-Watson process with finite variance survives for $R_k+k \le 2R_k$ generations. This is at least $c/R_k$ for some constant $c$ by Lemma \ref{kollem}. Thus the probability that \emph{none} of the $(z,k)$-frozen particles has a descendant that hits $z$ without taking any more steps away from $z$ is at most $(1-c/R_k)^{N_k}\le \exp(-cN_k/R_k)$. Therefore if $H(z)$ is the first hitting time of $z$, we have
\[\P_{N,0}(H(z)> R_k + 2k \,|\, \mathcal B_k) \le \exp(-cN_k/R_k).\]
Now, if a particle hits $z$ without taking more than $k$ steps away from $z$, then for every $x$ on the path from $0$ to $z$, $x$ is hit by time $d(0,x)+2k$. Thus
\begin{align}
\P_{N,0}(\exists r\le R_k : \tcov(r) > r + 2k \,|\,\mathcal B_k) &= \P_{N,0}(\exists r\le R_k, \, x\in \partial B(r) : H(x) > r + 2k \,|\, \mathcal B_k)\nonumber\\
&\le \P_{N,0}(\exists z\in \partial B(R_k) : H(z) > R_k + 2k \,|\, \mathcal B_k)\nonumber\\
&\le d(d-1)^{R_k-1}\exp(-cN_k/R_k).\label{tcovupperbd}
\end{align}
Then
\begin{align*}
&\P_{N,0}\Big(\exists k : \max_{r\le R_k}(\tcov(r)-r) > 2k\Big)\\
&\hspace{30mm}\le \P_{N,0}\Big(\bigcup_{j=1}^\infty \mathcal B_j^c\Big) + \P_{N,0}\Big(\{\exists k,\, r\le R_k : \tcov(r) > r + 2k\}\cap \bigcap_{j=1}^\infty \mathcal B_j\Big)\\
&\hspace{30mm}\le \P_{N,0}\Big(\bigcup_{j=1}^\infty \mathcal B_j^c\Big) + \sum_{k=1}^\infty \P_{N,0}(\{\exists r\le R_k : \tcov(r) > r + 2k\}\cap \mathcal B_k)\\
&\hspace{30mm}\le \P_{N,0}\Big(\bigcup_{j=1}^\infty \mathcal B_j^c\Big) + \sum_{k=1}^\infty \P_{N,0}(\exists r\le R_k : \tcov(r) > r + 2k \,|\, \mathcal B_k).
\end{align*}
By \eqref{Zunionupperbd} and \eqref{tcovupperbd}, this is at most
\[\sum_{j=2}^\infty d(d-1)^{R_j-1}e^{-\delta N_{j-1}^{1/2} / a} + \P_{N,0}(\mathcal B_1^c) + \sum_{k=1}^\infty d(d-1)^{R_k-1}\exp(-cN_k/R_k).\]
Recalling that for $k\ge1$
\[N_k = \frac{e^{(3/2)^k}}{\delta^2 a^2} \hspace{4mm} \hbox{ and } \hspace{4mm} R_k = \sum_{j=1}^{k-1} a N_j^{1/2},\]
we note that for any $\eps>0$, by choosing $a$ sufficiently small we can ensure that 
\[\sum_{j=2}^\infty d(d-1)^{R_j-1}e^{-\delta N_{j-1}^{1/2} / a} + \sum_{k=1}^\infty d(d-1)^{R_k-1}\exp(-cN_k/R_k) < \eps\]
and thus
\[\P_{N,0}\Big(\exists k : \max_{r\le R_k}(\tcov(r)-r) > 2k\Big) \le \P_{N,0}(\mathcal B_1^c) + \eps.\]
Since $R_1=0$ and $N_1 = e^{3/2}/(\delta^2 a^2)$, we have $\mathcal B_1 = \{F^{(1)}_0 + Y^{(1)}_0 \ge e^{3/2}/(\delta^2 a^2)\}$. However, under $\P_{N,0}$, we have $F^{(1)}_0 = 0$ and $Y^{(1)}_0 = N$, so for $N\ge e^{3/2}/(\delta^2 a^2)$ we have $\P_{N,0}(\mathcal B_1)=1$ and therefore
\[\P_{N,0}\Big(\exists k : \max_{r\le R_k}(\tcov(r)-r) > 2k\Big) < \eps.\]
In our original model we started with $1$ particle (rather than $N$ particles) at the origin, but by waiting until the first time at which the number of particles at $0$ is at least $e^{3/2}/(\delta^2a^2)$, which is almost surely finite given that the process survives, we may choose $t$ large enough such that
\[\P_{1,0}\Big(\exists k : \max_{r\le R_k}(\tcov(r)-r) > 2k+t \,\Big|\, \text{survival}\Big) < 2\eps.\]
Since $\eps>0$ was arbitrary, applying this with $k = \frac{(1+\eta)\log\log n}{\log(3/2)}$ for arbitrarily small $\eta>0$ completes the proof of the upper bound in Theorem \ref{mainthm}.
\qed

\section*{Acknowledgements}
Many thanks go to Itai Benjamini, who asked me this question, and sent me a short proof of a weaker bound on $\tcov(r)$ that he had written jointly with Gady Kozma. My thanks also to Alice Callegaro and an anonymous referee, each of whom suggested several corrections and improvements to the paper.

\vspace{2mm}

\noindent
This work was supported by a Royal Society University Research Fellowship.

\bibliographystyle{plain}
\def\cprime{$'$}

\end{document}